\documentclass[12pt]{article}

\usepackage{fullpage}
\usepackage{amssymb}
\usepackage{amsmath}
\usepackage{amsthm}
\usepackage{amscd}

\usepackage[english]{babel}
\theoremstyle{plain}
\newtheorem{theorem}{Theorem}
\newtheorem{lemma}{Lemma}
\theoremstyle{definition}
\newtheorem{defn}{Definition}
\newtheorem{example}{Example}

\makeatletter
\def\rdots{\mathinner{\mkern1mu\raise\p@
\vbox{\kern7\p@\hbox{.}}\mkern2mu
\raise4\p@\hbox{.}\mkern2mu\raise7\p@\hbox{.}\mkern1mu}}

\begin{document}

\title{A Method for Obtaining Generating Functions for Central Coefficients of Triangles}
		 
\date{}
\author{Dmitry Kruchinin\\
Vladimir Kruchinin\\
Tomsk State University of Control Systems and Radioelectronics,\\
Tomsk, Russian Federation\\
KruchininDm@gmail.com}
\maketitle

\begin{abstract}
We present techniques for obtaining a generating function for the central coefficients of a triangle $T(n,k)$, which is given by the expression $[xH(x)]^k=\sum_{n\geqslant k} T(n,k)x^n$, $H(0)\neq 0$.
We also prove certain theorems for solving direct and inverse problems.
\end{abstract}

\section{Introduction}

In \cite{KruCompositae} the second author introduced the notion of the \textit{composita} of a given ordinary generating function $G(x)=\sum_{n>0}g(n)x^n$. 
\begin{defn}
The \textit{composita} is the function of two variables defined by
\begin{equation}
\label{Fnk0}G^{\Delta}(n,k)=\sum_{\pi_k \in C_n}{g(\lambda_1)g(\lambda_2) \cdots g(\lambda_k)},
\end{equation}
where $C_n$ is the set of all compositions of an integer $n$, $\pi_k$ is the composition $\sum_{i=1}^k\lambda_i=n$ into $k$ parts exactly.
\end{defn}
Comtet \cite[ p.\ 141]{Comtet_1974} considered similar objects and identities for exponential generating functions, and called them potential polynomials. In this paper we consider the case of ordinary generating functions.

The generating function of the composita is equal to
\begin{equation}
[G(x)]^k=\sum_{n\geqslant k} G^{\Delta}(n,k)x^n.
\end{equation}

For instance, we obtain the composita of the generating function $G(x,a,b)=ax+bx^2$.

The binomial theorem yields 

$$[G(x,a,b)]^k=x^k(a+bx)^k=x^k\sum_{m=0}^k \binom{k}{m}a^{k-m}b^mx^m. 
$$

Substituting $n$ for $m+k$, we get the following expression:
$$[G(x,a,b)]^k=\sum_{n=k}^{2k} \binom{k}{n-k}a^{2k-n}b^{n-k}x^n=\sum_{n=k}^{2k}G^{\Delta}(n,k,a,b)x^n. 
$$
Therefore, the composita is 
\begin{equation}
\label{Gnkab}
G^{\Delta}(n,k,a,b)=\binom{k}{n-k}a^{2k-n}b^{n-k}. 
\end{equation}

Now we show the compositae of several known generating functions \cite{Comtet_1974, Wilf_1994} in the Table \ref{tab:a}.

\begin{table}[h]
\begin{center}
\setlength\arrayrulewidth{1pt}
\renewcommand{\arraystretch}{1,3}
%\begin{tabular*}{0.9\linewidth}{|c @{\extracolsep{\fill}}c|}
\begin{tabular}{|ccc|ccc|}
\hline
&\textbf{Generating function $G(x)$}& & & \textbf{Composita $G^{\Delta}(n,k)$}&\\
\hline
&$ax+bx^2$ &&& $a^{2k-n}b^{n-k}\binom{k}{n-k}$&\\ \hline
&$\frac{bx}{1-ax}$ &&&$\binom{n-1}{k-1}a^{n-k}b^k$&\\  \hline
&$\ln(1+x)$ &&& $\frac{k!}{n!}\genfrac{[}{]}{0pt}{}{n}{k}$& \\ \hline
&$e^x-1$  &&& $\frac{k!}{n!}\genfrac{\{}{\}}{0pt}{}{n}{k}$& \\ \hline
\end{tabular}
\caption{Examples of generating functions and their compositae}
\label{tab:a}
\end{center}
\end{table}

The notation $\genfrac{[}{]}{0pt}{}{n}{k}$ are the Stirling numbers of the first kind(see \cite{Comtet_1974,ConcreteMath}). The Stirling numbers of the first kind  count the number of permutations of $n$ elements with $k$ disjoint cycles.

The Stirling numbers of the first kind are defined by the following generating function:
$$
\psi_k(x)=\sum_{n\geq k} \genfrac{[}{]}{0pt}{}{n}{k} \frac{x^n}{n!}=\frac{1}{k!}\ln^k(1+x).
$$

The notation $\genfrac{\{}{\}}{0pt}{}{n}{k}$ are the Stirling numbers of the second kind( see \cite{Comtet_1974,ConcreteMath}). The Stirling numbers of the second kind count the number of ways to partition a set of $n$ elements into $k$ nonempty subsets.

A general formula for the Stirling numbers of the second kind is given as follows:
$$\genfrac{\{}{\}}{0pt}{}{n}{k}=\frac{1}{k!}\sum_{j=0}^k(-1)^{k-j}\binom{k}{j}j^n.$$

The Stirling numbers of the second kind are defined by the generating function
$$
\Phi_k(x)=\sum_{n\geq k} \genfrac{\{}{\}}{0pt}{}{n}{k} \frac{x^n}{n!}=\frac{1}{k!}(e^x-1)^k.
$$

\section{Main results}
In this section we present the main results of the paper:
\begin{itemize}
\item The method of obtaining a generating function for central coefficients of the given triangle (Theorem \ref{thm1}).
\item Inverse problem. The method of obtaining a unique triangle, when we know only  the generating function for the central coefficients of this triangle (Theorem \ref{thm2}).

\end{itemize}
In tabular form, the composita is presented as a triangle as follows:
$$
\begin{array}{ccccccccccc}
&&&&& G_{1,1}^{\Delta}\\
&&&& G_{2,1}^{\Delta} && G_{2,2}^{\Delta}\\
&&& G_{3,1}^{\Delta} && G_{3,2}^{\Delta} && G_{3,3}^{\Delta}\\
&& G_{4,1}^{\Delta} && G_{4,2}^{\Delta} && G_{4,3}^{\Delta} && G_{4,4}^{\Delta}\\
& \rdots && \vdots && \vdots && \vdots && \ddots\\
G_{n,1}^{\Delta} && G_{n,2}^{\Delta} && \ldots && \ldots && G_{n,n-1}^{\Delta} && G_{n,n}^{\Delta}\\
\end{array}
$$

Considering the triangle, we see that the central coefficients are represented by the following sequence:

$$
G_{1,1}^{\Delta},\,G_{3,2}^{\Delta},\,\ldots,\,G_{2n-1,n}^{\Delta},\,\ldots
$$

The generating function for the central coefficients  is given as follows:

$$
F(x)=G_{1,1}^{\Delta}\,+\,G_{3,2}^{\Delta}x^1\,+\,\cdots\,+\,G_{2n-1,n}^{\Delta}x^{n-1}\,+\, \cdots
$$

In the following lemma we give the Lagrange inversion formula, which was proved by Stanley \cite{Stanley_1999}.

\begin{lemma}[The Lagrange inversion formula]
\label{Lagrange formula}
Suppose $H(x)=\sum_{n\geq 0}h(n)x^n$ with $h(0)\neq 0$, and let $A(x)$ be defined by
\begin{equation}
A(x)=xH(A(x)).
\end{equation}
Then
\begin{equation}
n[x^n]A(x)^k=k[x^{n-k}]H(x)^n,
\end{equation}
where $[x^n]A(x)^k$ is the coefficient of $x^n$ in  $A(x)^k$ and  $[x^{n-k}]H(x)^n$ is the coefficient of $x^{n-k}$ in  $H(x)^n$.
\end{lemma}

By using the above Lemma \ref{Lagrange formula}, we now give our main theorems.

\begin{theorem}
\label{thm1} Suppose  $H(x)=\sum_{n\geq 0}h(n)x^n$ is a generating function, where $h(0)\neq 0$, and  $G(x)=xH(x)$ with $[G(x)]^k=\sum_{n\geqslant k} G^{\Delta}(n,k)x^n$, and $A(x)=\sum_{n>0}a(n)x^n$ is the generating function, which is obtained from the functional equation $A(x)=xH(A(x))$. Then the generating function  $F(x)$ for the central coefficients of the triangle $G^{\Delta}(n,k)$ is equal to the first derivative of the function $A(x)$: 
\begin{equation}
F(x)=A'(x)=\sum_{n\geqslant 1} G^{\Delta}(2n-1,n)x^{n-1}.
\end{equation}
\end{theorem}

\begin{proof} According to Lemma \ref{Lagrange formula}, for the solution of the functional equation $A(x)=xH(A(x))$, we can write
$$
n[x^n]A(x)^k=k[x^{n-k}]H(x)^n.
$$

In the left-hand side, there is the composita of the generating function $A(x)$ multiplied by $n$:
$$
n[x^n]A(x)^k=n\,A^{\Delta}(n,k).
$$

We know that
$$
\left( xH(x)\right)^k=\sum_{n\geq k}G^{\Delta}(n,k)x^n.
$$
Then

$$
\left( H(x)\right)^k=\sum_{n\geq k}G^{\Delta}(n,k)x^{n-k}.
$$
If we replace $n-k$ by $m$, we obtain the following expression:

$$
\left( H(x)\right)^k=\sum_{m\geq 0}G^{\Delta}(m+k,k)x^{m}.
$$
Substituting $n$ for $k$ and $n-k$ for $m$, we get
$$
[x^{n-k}]H(x)^n=G^{\Delta}(2n-k,n).
$$
Therefore, we get (cf. \cite{KruCompositae})
\begin{equation}
A^{\Delta}(n,k)=\frac{k}{n}G^{\Delta}(2n-k,n).
\end{equation}

Hence, for solutions of the functional equation $A(x)=xH(A(x))$, we can use the following expression: 
$$
[A(x)]^k=\sum_{n\geqslant k}A^{\Delta}(n,k)x^n=\sum_{n\geqslant k} \frac{k}{n}G^{\Delta}(2n-k,n)x^n.
$$
Therefore,
$$
A(x)=\sum_{n\geqslant 1} \frac{1}{n}G^{\Delta}(2n-1,n)x^n.
$$

Differentiating $A(x)$ with respect to $x$, we obtain the required generating function:

$$
F(x)=A'(x)=\sum_{n\geqslant 1} G^{\Delta}(2n-1,n)x^{n-1}.
$$

The theorem is proved.
\end{proof}

For applications of Theorem \ref{thm1} we give some examples.
\begin{example} Consider the expression that generates Pascal's triangle
$$
[xH(x)]^k=\left(\frac{x}{1-x}\right)^k=\sum_{n\geqslant k} \binom{n-1}{k-1}x^n,
$$
$$
\begin{array}{ccccccccccc}
&&&&& 1\\
&&&& 1 && 1\\
&&& 1 && 2 && 1\\
&& 1 && 3 && 3 && 1\\
& 1 && 4 && 6 && 4  && 1 \\
\end{array}
$$

Let us obtain a generating function for the well-known central binomial coefficients A000984 \cite{oeis,Lehmer}:
$$
1, 2, 6, 20, 70, 252, 924, 3432, 12870, 48620, \ldots
$$

The $n^\text{th}$ central binomial coefficient is defined by
$$\binom{2n}{n}.$$ 

The functional equation has the form
$$
A(x)=\frac{x}{1-A(x)}.
$$

The solution of this equation is
$$
A(x)=\frac{1-\sqrt{1-4x}}{2}.
$$

Now we compute the derivative of $A(x)$. Therefore, we obtain the generating function for the central binomial coefficients:
\begin{equation}
A'(x)=\frac{1}{\sqrt{1-4x}}
\end{equation}

See A000984 \cite{oeis}.
\end{example}

\begin{example} Let us consider the generating 
function $H(x)=\frac{1-x}{1-2x}$. Then
$$
[xH(x)]^k=\sum_{n\geqslant k} G^{\Delta}(n,k)x^n,
$$
where
$$
G^{\Delta}(n,k)=\sum_{i=0}^{n-k}{\left(-2\right)^{i}\, \binom{k}{n-k-i}\, \binom{k+i+1}{k-1}\,\left(-1\right)^{n-k-i}}.
$$
The composita $G^{\Delta}(n,k)$ defines the following triangle A105306 \cite{oeis}
$$
\begin{array}{ccccccccccc}
&&&&& 1\\
&&&& 1 && 1\\
&&& 2 && 2 && 1\\
&& 4 && 5 && 3 && 1\\
& 8 && 12 && 9 && 4  && 1 \\
\end{array}
$$

Let us now compute the generating function for the central coefficients of this triangle. For this purpose we solve the following functional equation:
$$
A(x)=x\frac{1-A(x)}{1-2A(x)}.
$$

Therefore,
$$
A(x)=\frac{1+x-\sqrt{1-6x+x^2}}{4}
$$
which generates the sequence A001003 \cite{oeis}
$$x+x^2+3\,x^3+11\,x^4+45\,x^5+197\,x^6+903\,x^7+\cdots $$

Now we compute the derivative of $A(x)$. Therefore, we obtain the generating function for the central coefficients:
\begin{equation}
A'(x)=\frac{1}{4}-\frac{x-3}{4\,\sqrt{x^2-6\,x+1}}
\end{equation}

This is the sequence A176479 \cite{oeis}
$$1+2\,x+9\,x^2+44\,x^3+225\,x^4+1182\,x^5+\cdots $$
\end{example}

\begin{example} Let us consider the generating 
function $H(x)=x\cot(x)$. Then $[xH(x)]^k=[x^2\cot(x)]^k$ defines the following triangle A199542 \cite{oeis}

$$G^{\Delta}(n,k)=2^{n-2\,k}\,\left(-1\right)^{\frac{n-k}{2}}\,\sum_{l=0}^{k} 2^l\,l!\, \binom{k}{l}\,
\sum_{m=0}^{n-2\,k+l}\frac{m!\,
\genfrac{[}{]}{0pt}{}{l+m}{l}\,\genfrac{\{}{\}}{0pt}{}{n-2k+l}{m}}{\left(l+m\right)!(n-2\,k+l)!},$$

$$
\begin{array}{ccccccccccc}
&&&&& 1\\
&&&& 0 && 1\\
&&& -\frac{1}{3} && 0 && 1\\
&& 0 && -\frac{2}{3} && 0 && 1\\
& -\frac{1}{45} && 0 && -1 && 0  && 1 \\
\end{array}
$$

The functional equation has the form

$$
A(x)=xA(x)\cot(A(x)).
$$

Then
$$
A(x)=\arctan(x).
$$

Therefore, the generating function for the central coefficients of the triangle A199542 \cite{oeis} is
\begin{equation}
A'(x)=\frac{1}{1+x^2}.
\end{equation}
\end{example}

Now we are ready to consider the inverse problem. We know the generating function for the central coefficients, and we need to find a triangle with given central coefficients.

First we consider the notion of \textit{reciprocal generating functions} \cite{Wilf_1994}.
\begin{defn}
Reciprocal generating functions are functions that satisfy the condition:
$$
H(x)B(x)=1. 
$$
\end{defn}

In the following lemma we give the formula of the composita of a reciprocal function, which was proved in the paper \cite{KruCompositae}.

\begin{lemma}
\label{recip_Theorem}
Suppose  $H(x)=\sum_{n\geq 0}h(n)x^n$ is a generating function, where $h(0)=1$, and $B(x)=\sum_{n\geq 0}b(n)x^n$ is the reciprocal generating function, where $b(0)=1$, and $B_x^{\Delta}(n,m)$ is the composita of $xB(x)$. Then the composita of the function $xH(x)$ is equal to

\begin{equation}
H_x^{\Delta}(n,k)=
\begin{cases}
1, & if \text{~}\text{~} n=k;\\
\sum\limits_{m=1}^{n-k}(-1)^{m}\binom{k+m-1}{k-1}
 \sum\limits_{j=1}^{m}(-1)^{j}\binom{m}{j}B_x^{\Delta}(n-k+j,j), & if \text{~}\text{~} n>k.
\end{cases}
\label{Reciprocal}
\end{equation}

\end{lemma}

By using the above Lemma \ref{recip_Theorem}, we are now able to prove the following theorem.

\begin{theorem}
\label{thm2}
Suppose  $F(x)=\sum_{n\geq 0}f(n)x^n$ is a generating function, where $F(0)\neq 0$. Then there exists a unique generating function $H(x)=\sum_{n\geq 0}h(n)x^n$ such that the triangle $G^\Delta(n,k)$ defined below has the central coefficients $f(n)$, 
$$
[xH(x)]^k=\sum_{n\geqslant k} G^\Delta(n,k)x^n.
$$
\end{theorem}

\begin{proof}  To prove this theorem, we consider the proof of Theorem \ref{thm1} in the reverse order. 

Integrating $F(x)$ in $x$, we obtain the generating function $A(x)$
$$
A(x)=\int F(x)dx=\sum_{n>0 } f(n-1) \frac{x^n}{n},
$$ 
where $a(0)=0$.

To find $H(x)$, we solve the following inverse functional equation
\begin{equation}
A(x)=xH(A(x)).
\end{equation}

Suppose $A(x)=t$, then $x=A^{-1}(t)$. And we get

$$
t=A^{-1}(t)H(t)
$$
or
\begin{equation}
H(t)=\frac{t}{A^{-1}(t)}.
\end{equation}

We solve this equation using the notion of compositae.

The generating function $B(x)=\frac{A^{-1}(t)}{t}$ is the reciprocal generating function with respect to $H(t)$
$$
H(t)\frac{A^{-1}(t)}{t}=1.
$$

According to Lemma \ref{recip_Theorem}, the composita of $tH(t)$ is equal to

$$
G^{\Delta}(n,k)=\begin{cases}
1, & if \text{~}\text{~} n=k;\\
\sum\limits_{m=1}^{n-k}(-1)^{m}\binom{k+m-1}{k-1}
 \sum\limits_{j=1}^{m}(-1)^{j}\binom{m}{j}B_t^{\Delta}(n-k+j,j), & if \text{~}\text{~} n>k.
\end{cases}
$$

In our case, the composita $B_t^{\Delta}(n,k)$ is equal to the composita of $A^{-1}(t)$.

For obtaining the composita of $A^{-1}(t)$, we use the following algorithm:

\begin{enumerate}
\item calculate the composita $A^\Delta(n,k)$ of the generating function $A(x)$;
\item using formula (\ref{Reciprocal}), calculate the reciprocal composita $A^\Delta_R(n,k)$  of the composita of $A^\Delta(x)$;
\item using  the reciprocal composita  $A^{\Delta}_R(n,k)$, calculate the composita of  the inverse function $A^{-1}(t)$
$$
A_{Inv}^{\Delta}(n,k)=\frac{k}{n}A_R^{\Delta}(2n-k,n).
$$
\end{enumerate}

Therefore,
$$H(x)=\sum_{n\geqslant 1} G^\Delta(n,1)x^{n-1}.$$

The theorem is proved.
\end{proof}

Using Theorem \ref{thm2}, we can obtain a triangle $T(n,k)$ given by the generating function $H(x)$, $h(0)\neq 0$ such that 
$$
[xH(x))]^k=\sum\limits_{n\geqslant k} T(n,k)x^n,
$$
when we know only the central coefficients of the triangle.

Now we give the following example. We find a triangle $T(n,k)$ such that a central coefficients are 
the Catalan numbers.

\begin{example} The generating function for the Catalan numbers is defined by A000108 \cite{oeis}
$$
F(x)=\frac{1-\sqrt{1-4x}}{2x}.
$$

Let us obtain the required triangle and its expression.
First we compute the integral of $F(x)$
$$
\int \frac{1-\sqrt{1-4x}}{2x}dx= \log\left(\sqrt{1-4\,x}+1\right)-\sqrt{1-4\,x}.
$$

The first elements of the obtained generating function are shown below:
$$-1+\log 2+x+{\frac{x^2}{2}}+{\frac{2\,x^3}{3}}+{\frac{5\,x^4}{4}}+
 {\frac{14\,x^5}{5}}+7\,x^6+{\frac{132\,x^7}{7}}+\cdots $$
 
Since $a(0)=0$, we have
$$
A(x)=1-\log(2)+\log\left(\sqrt{1-4\,x}+1\right)-\sqrt{1-4\,x}=
$$
$$
=\log\left(1-\frac{1-\sqrt{1-4\,x}}{2}\right)+2\left(\frac{1-\sqrt{1-4\,x}}{2}\right).
$$

Now we obtain the composita of $A(x)$.
$$
A(x)=\log(1-C(x))+2C(x),
$$
where 
$$C(x)=\frac{1-\sqrt{1-4\,x}}{2}.$$

The composita of $\log(1-x)+2x$ is 
$$\sum_{j=0}^{k}{{{\left(-1\right)^{n-j}\,2^{j}\,\binom{k}{j}\,\frac{\left(k-j\right)!}{{\left(n
 -j\right)!}}\, \genfrac{[}{]}{0pt}{}{n-j}{k-j}}}}.$$
 
The composita of $C(x)$ is
$${\frac{k}{n}\, \binom{2\,n-k-1}{n-1}}.$$

Therefore, the composita of $A(x)$ is equal to the product of compositae $C(x)$ and $\log(1-x)+2x$
$$
A^{\Delta}(n,k)=\sum_{m=k}^n {{\frac{m}{n}\,\binom{2\,n-m-1}{n-1}}} \sum_{j=0}^{k}{{{\left(-1\right)^{m-j}\,2^{j}\,\binom{k}{j}\,\frac{\left(k-j\right)!}{{\left(m
 -j\right)!}}\,\genfrac{[}{]}{0pt}{}{m-j}{k-j}}}}.
$$

Using formula (\ref{Reciprocal}), the reciprocal composita of the composita of $A(x)$ is
$$
A_R^{\Delta}(n,k)=\begin{cases}
1,& if \text{~}\text{~} n=k;\\
\sum\limits_{m=1}^{n-k}\binom{k+m-1}{k-1}\,\sum\limits_{j=1}^{m}(-1)^{j}\,\binom{m}{j}\,A^{\Delta}(n-k+j,j),& if \text{~}\text{~} n>k.
\end{cases}
$$

Using  the reciprocal composita  $A_R^{\Delta}(n,k)$, we obtain the composita of  the inverse function $[A(x)]^{-1}$
$$
A_{Inv}^{\Delta}(n,k)=\frac{k}{n}A^{\Delta}(2n-k,n).
$$

Therefore, we obtain the reciprocal composita of the composita of the inverse function
$$
G^{\Delta}(n,k)=\begin{cases}
1,& if \text{~}\text{~} n=k;\\
\sum\limits_{m=1}^{n-k}\binom{k+m-1}{k-1}\,\sum\limits_{j=1}^{m}(-1)^{j}\,\binom{m}{j}\,A_{Inv}^{\Delta}(n-k+j,j),& if \text{~}\text{~} n>k.
\end{cases}
$$
This is the required composita.

The required triangle has the following form
$$
\begin{array}{ccccccccccccccccc}
&&&&&&&& 1 \\
&&&&&&&\frac{1}{2} && 1 \\
&&&&&&\frac{5}{12}&&  1 &&  1  \\
&&&&& \frac{1}{2} && \frac{13}{12} && \frac{3}{2} && 1\\  
&&&&\frac{551}{720} && \frac{17}{12} && 2 && 2 && 1  \\
&&&\frac{11}{8} && \frac{529}{240} && \frac{23}{8} && \frac{19}{6} && \frac{5}{2} && 1\\ 
&&\frac{16657}{6048} && \frac{2831}{720} && \frac{1111}{240} && 5 && \frac{55}{12} && 3 && 1\\  
&\frac{4289}{720} && \frac{46999}{6048} && \frac{1329}{160} && \frac{6059}{720} && \frac{95}{12} && \frac{25}{4} && \frac{7}{2} && 1\\  
\frac{16491599}{1209600} && \frac{501353}{30240} && \frac{246787}{15120} && \frac{1841}{120} && 14 && \frac{47}{4} && \frac{49}{6} && 4 && 1\\  
 \end{array}
 $$

\end{example}

Concerned with sequences A000108, A000984, A001003, A105306, A176479 and A199542.

2000 Mathematics Subject Classification: 
Primary 05A15; Secondary 11B75, 05A10.

Keywords: generating function, central coefficients, triangle, composita.
\end{document}